\documentclass[pdflatex,sn-mathphys-num]{sn-jnl}

\usepackage{amsmath}
\usepackage{amsthm}
\usepackage{amssymb}
\usepackage{mathtools}
\usepackage{enumitem}
\usepackage{graphicx}
\usepackage[dvipsnames]{xcolor}
\usepackage{tikz}
\usetikzlibrary{arrows.meta}
\usepackage{hyperref}

\hypersetup{
  colorlinks=true,
  linkcolor=MidnightBlue,
  citecolor=BrickRed,
  urlcolor=MidnightBlue,
  pdfauthor={A. Bahayou},
  pdftitle={Toric geometry of Schouten squares on rotational double extensions}
}

\numberwithin{equation}{section}

\newtheorem{theorem}{Theorem}[section]
\newtheorem{lemma}[theorem]{Lemma}
\newtheorem{proposition}[theorem]{Proposition}
\newtheorem{corollary}[theorem]{Corollary}
\theoremstyle{definition}
\newtheorem{definition}[theorem]{Definition}
\newtheorem{example}[theorem]{Example}
\theoremstyle{remark}
\newtheorem{remark}[theorem]{Remark}

\newcommand{\g}{\mathfrak{g}}
\newcommand{\fb}{\mathfrak{b}}
\newcommand{\fz}{\mathfrak{z}}
\newcommand{\Rone}{\mathcal{R}^{(1)}}
\newcommand{\Rzero}{\mathcal{R}^{0}}
\newcommand{\Rplus}{\mathcal{R}^{+}}
\DeclareMathOperator{\Aut}{Aut}
\DeclareMathOperator{\Newt}{Newt}
\DeclareMathOperator{\conv}{conv}
\newcommand{\spanQ}{\operatorname{span}_{\mathbb{Q}}}
\newcommand{\spanZ}{\operatorname{span}_{\mathbb{Z}}}
\newcommand{\spanR}{\operatorname{span}_{\mathbb{R}}}
\DeclareMathOperator{\rk}{rank}
\DeclareMathOperator{\tr}{tr}
\DeclareMathOperator{\diag}{diag}
\DeclareMathOperator{\pr}{pr}
\newcommand{\hh}{\widehat{h}}

\title[Toric geometry of Schouten squares]{Toric geometry of Schouten squares on rotational double extensions}
\author*[1]{\fnm{A.} \sur{Bahayou}}\email{bahayou.amine@univ-ouargla.dz}

\affil*[1]{\orgdiv{Laboratory of Mathematics and Applications},
\orgname{Kasdi Merbah University},
\orgaddress{\postcode{30000}, \city{Ouargla}, \country{Algeria}}}

\begin{document}

\abstract{Let $\g=\fb\oplus V\oplus\fb^{*}$ be a Medina-Revoy rotational double extension in which an abelian Lie algebra $\fb$ acts by scalar rotations on an orthogonal sum of oriented Euclidean two-planes. We analyse the restriction of the Schouten square $r\mapsto[r,r]$ to the block-rank-one locus $\Rone\subset\fb\wedge V$ and to its nondegenerate open part $\Rplus$. On every marked slice $S_{H}\subset\Rplus$ the Schouten square splits orthogonally into a radial and an angular component. The radial component factors through the standard toric moment map: its image is a simplicial cone, it satisfies a sharp quadratic coercive estimate with an intrinsic optimal constant, and the restricted Schouten map is proper with compact semialgebraic affine fibres. The angular component becomes Laurent monomial in joint eigencoordinates; a marker-cancellation criterion determines its signed exponent configuration, whose integral lattice controls the differential rank, the compact isotropy, its component group, and the Laurent binomial ideal defining the Zariski closure of the complexified image. Smith normal form makes each invariant algorithmic. All statements hold for arbitrary $\dim\fb$ and impose no nonresonance assumption. A corollary of the underlying disjoint-support decomposition classifies the classical Yang-Baxter equation on $\Rone$: for any bivector whose $\fb\wedge V$ component lies in $\Rone$, triangularity forces every block detected by its rotation weight to vanish, regardless of the $\Lambda^{2}\fb$, $\Lambda^{2}V$, or central-wedge terms.}
\maketitle

\section{Introduction}\label{sec:intro}

\subsection{The problem and its geometric setting}\label{sec:problem}

The classical Yang-Baxter equation
\begin{equation}\label{eq:cybe}
[r,r]=0,\qquad r\in\Lambda^{2}\g,
\end{equation}
is the algebraic constraint controlling triangular coboundary Lie bialgebras, triangular Poisson-Lie structures, and homogeneous Yang-Baxter deformations of sigma and Wess-Zumino-Witten models \cite{Drinfeld1983,SemenovTianShansky1983,LuWeinstein1990,KosmannSchwarzbach2004,ChariPressley1994,BorsatoWulff2019,HoareLacroix2020,Hoare2022}. For semisimple targets the Belavin-Drinfeld theorem organises solutions by root data \cite{BelavinDrinfeld1982}; on solvable quadratic Lie algebras no comparable framework exists, largely because the polynomial map $r\mapsto[r,r]$ couples several exterior supports and its image and fibres are seldom directly accessible.

This paper isolates a natural stratum on which the map decouples completely and becomes governed by toric geometry. Let
\begin{equation}\label{eq:g-decomp}
\g=\fb\oplus V\oplus\fb^{*}
\end{equation}
be the Medina-Revoy quadratic double extension \cite{MedinaRevoy1985,Bordemann1997} determined by an abelian Lie algebra $\fb$ acting by scalar rotations on an orthogonal sum $V=\bigoplus_{\ell=1}^{m}P_{\ell}$ of oriented Euclidean two-planes. Such algebras form a rotational subfamily of the metric extensions of Kath-Olbrich \cite{KathOlbrich2004,KathOlbrich2006}, of the general Manin-triple constructions surveyed in \cite{Ovando2016,DiattaMedina2004}, and encompass oscillator and Nappi-Witten algebras \cite{NappiWitten1993,BoucettaMedina2011,BenitoEtAl2024}. The mechanism at the heart of this paper is the planar identity

\[
u\wedge Ju=\|u\|^{2}\Omega \qquad (u\in P,\ J=\text{quarter-turn},\ \Omega=\text{area}),
\]
which, applied to a decomposable block $h\wedge u\in\fb\wedge P$, forces $[h\wedge u,h\wedge u]$ to a fixed ray in $\fb\wedge\Lambda^{2}P$. Different two-planes carry disjoint exterior supports; cross-plane brackets deposit only angular information. One Schouten calculation therefore separates cleanly into a radial toric moment map and an angular monomial map.

\subsection{Positive block-rank-one locus}\label{sec:locus}

Under the canonical identification $\fb\wedge P_{\ell}\cong\fb\otimes P_{\ell}$, consider bivectors whose projection to each block has tensor rank at most one. Every nonzero block $p_{\ell}=h_{\ell}\wedge u_{\ell}$ then determines an intrinsic marker line $\mathbb Rh_{\ell}\subset\fb$ and a free phase circle in $P_{\ell}$; when the marker escapes $\ker\lambda_{\ell}$ it acquires a canonical orientation via $\lambda_{\ell}(h_{\ell})>0$. Fixing the marker lines produces the \emph{marked slices}
\begin{equation}\label{eq:slice}
S_{H}=\bigoplus_{\ell=1}^{m}h_{\ell}\wedge P_{\ell},
\qquad \lambda_{\ell}(h_{\ell})>0.
\end{equation}
These slices form a semialgebraic cover of the positive block-rank-one locus $\Rplus$; on each fixed-support subset the marker rays are intrinsic, while along boundary faces where a plane component vanishes distinct slices meet.

The rank-one condition is not a technical convenience but a maximal one. In an oriented orthonormal basis $(e_{\ell},f_{\ell})$ a general block $x_{\ell}\wedge e_{\ell}+y_{\ell}\wedge f_{\ell}$ produces

\[
[p_{\ell},p_{\ell}]=2\bigl(\lambda_{\ell}(x_{\ell})x_{\ell}+\lambda_{\ell}(y_{\ell})y_{\ell}\bigr)\wedge\Omega_{\ell}
+2\lambda_{\ell}\wedge x_{\ell}\wedge y_{\ell},
\]
whose first summand no longer sits on a fixed ray and can vanish on nontrivial rank-two data, and whose second summand couples the block to the central sector $\fz\wedge\fb\wedge V$. Both effects break the simplicial moment cone and the coercive estimate simultaneously. The block-rank-one condition is thus the exact locus on which the toric picture holds; §\ref{sec:conclusion} makes precise how each ingredient degenerates beyond it.

\subsection{Main results}\label{sec:principal}

For $r=\sum_{\ell}h_{\ell}\wedge u_{\ell}\in S_{H}$ the planar identity gives the exact orthogonal decomposition
\begin{equation}\label{eq:decomp-intro}
[r,r]=\Phi_{H}(r)+\mu_{H}(r),\qquad
\Phi_{H}(r)=\sum_{\ell=1}^{m}\|u_{\ell}\|^{2}E_{\ell}(H),
\end{equation}
with radial generators $E_{\ell}(H)=2\lambda_{\ell}(h_{\ell})\,h_{\ell}\wedge\Omega_{\ell}$ occupying pairwise disjoint exterior supports. Three theorems organise the paper.

\smallskip
\noindent\textbf{Theorem A} (radial toric geometry; Theorem~\ref{thm:structure}).
\emph{The radial map $\Phi_{H}$ is an injective linear rescaling of the standard toric moment map for the $\mathbb T^{m}$-action by independent plane rotations. Its image is the simplicial cone
\[
C_{H}=\bigoplus_{\ell}\mathbb R_{\ge0}E_{\ell}(H),
\]
and its fibres are products of circles. For every positive-definite auxiliary metric on $\fb$,}
\begin{equation}\label{eq:coercive-intro}
\|\Phi_{H}(r)\|\ge \kappa_{0}(H)\|r\|^{2},\qquad
\kappa_{0}(H)=2\left(\sum_{\ell=1}^{m}
\frac{\|h_{\ell}\|^{2}}{\lambda_{\ell}(h_{\ell})^{2}}\right)^{-1/2},
\end{equation}
\emph{with equality on an explicit weighted-barycentric locus. The restricted Schouten map $\Sigma_{H}=[\,\cdot\,,\cdot\,]|_{S_{H}}$ is proper, its image is closed and semialgebraic, and every affine fibre is compact.}

\smallskip
\noindent\textbf{Theorem B} (monomial angular geometry; Theorem~\ref{thm:monomial}).
\emph{Fixing a radial value freezes the moduli $\|u_{\ell}\|$ and leaves a product of phase circles. In joint eigencoordinates the angular map is Laurent monomial,}
\[
(s_{\ell})_{\ell\in I}\longmapsto
\bigl(a_{\alpha}(H,Y)s^{\alpha}\bigr)_{\alpha\in A_{H,Y}},
\]
\emph{with signed exponent configuration $A_{H,Y}\subset\mathbb Z^{I}$ determined by the marker-cancellation criterion. Its integral lattice $L_{A}=\spanZ(A_{H,Y})$ controls the differential rank, the compact isotropy $K_{A}\subset\mathbb T^{I}$, and its component group $L_{A}^{\mathrm{sat}}/L_{A}$; the Zariski closure of the complexified image is the translated subtorus defined by the Laurent binomials associated with $\ker_{\mathbb Z}M_{A}$, and Smith normal form makes every invariant algorithmic.}

\smallskip
\noindent\textbf{Theorem C} (rank-one support rigidity; Theorem~\ref{thm:obstruction}).
\emph{Let $r\in\Lambda^{2}\g$ with $\fb\wedge V$-component $p\in\Rone$. If $[r,r]=0$, then $p\in\Rzero$; that is, every nonzero block $p_{\ell}=h_{\ell}\wedge u_{\ell}$ satisfies $\lambda_{\ell}(h_{\ell})=0$. In particular, if $p\in\Rplus$, then $p=0$, and this conclusion is insensitive to arbitrary $\Lambda^{2}\fb$, $\Lambda^{2}V$, and central-wedge components of $r$.}

\smallskip
Theorems A and B are the principal structural results and give the complete toric picture of the restricted Schouten map. Theorem C is logically parallel to, rather than dependent on, Theorem B: both are consequences of the disjoint-support mechanism underlying Theorem A. Its proof rests on the planewise identity \eqref{eq:planewise1} and on support-tracking of the $\Lambda^{2}\fb$, $\Lambda^{2}V$, and central-wedge components of $r$; it requires no eigencoordinates. Its ambient class is correspondingly larger than the marked slices analysed for A and B: the complementary components of $r$ are unconstrained. When $\dim\fb=1$ the rank-one condition is automatic and $\ker\lambda_{\ell}=0$, so every triangular tensor on an oscillator or Nappi-Witten algebra has zero $\fb\wedge V$ component, at every frequency pattern including resonance.

\subsection{Position in the literature}\label{sec:literature}

The moment-map and algebraic-torus tools employed below are classical \cite{Atiyah1982,GuilleminSternberg1982,CoxLittleSchenck2011,EisenbudSturmfels1996}; their role here is diagnostic and computational, exposing toric structure produced by the Schouten calculation itself. The new mathematical content lies in the disjoint-support decomposition \eqref{eq:decomp-intro}, the sharp radial estimate \eqref{eq:coercive-intro} with its intrinsic barycentric equality locus, the marker-cancellation criterion, and its translation into a finite lattice datum determining every fibre invariant.

For oscillator algebras with $\dim\fb=1$, Boucetta and Medina \cite{BoucettaMedina2011} classified triangular $r$-matrices under generic frequency hypotheses; related bialgebra and Poisson structures appear in \cite{BallesterosHerranz1996,BallesterosHerranz1997,AlbuquerqueEtAl2021}. Delorme's classification of Manin triples on quadratic Lie algebras \cite{Delorme2001} and further quantization results \cite{EtingofKazhdan1996,AndruskiewitschJancsa2015} provide the general Lie-bialgebraic context. In the specific $H_{4}=NW_{4}$ setting, the classification of inequivalent (modified) classical $r$-matrices and the calculation of the corresponding deformed backgrounds have been carried out in \cite{DemulderEtAl2021} using distinct methods. The rotational Medina-Revoy construction on which we operate belongs to the Kath-Olbrich framework of quadratic extensions \cite{KathOlbrich2004,KathOlbrich2006,Ovando2016}, and recent work on skew-adjoint endomorphisms and generalized oscillator algebras \cite{BenitoEtAl2024} provides a complementary structural viewpoint. The author's earlier work \cite{Bahayou2023,Bahayou2023arxiv} addresses the flat quotient $\fb\ltimes V$; here the quadratic central extension $[V,V]\subset\fb^{*}$ is retained and is essential to the exterior-support argument for Theorem~\ref{thm:obstruction}. Root-theoretic techniques for rational solutions on simple Lie algebras appear in \cite{Stolin1991}.

The present contribution differs from the classification programme along a controlled axis: it determines the image and fibres of the restricted polynomial map $[\,\cdot\,,\cdot\,]|_{S_H}$ and, as a corollary of the underlying disjoint-support decomposition, removes every genericity assumption from the vanishing of the $h\wedge V$ support of triangular tensors. Resonant equal-frequency blocks, excluded by the classical distinct-frequency hypothesis, are precisely where the exponent set $A_{H,Y}$ drops rank; Example~\ref{ex:resonant} exhibits this cancellation stratum. Within the oscillator setting our conclusion is the support projection of Boucetta-Medina's classification; on higher-dimensional rotational extensions it is new.

\subsection{Physical motivation}

Homogeneous Yang-Baxter deformations of WZW and sigma models take triangular skew tensors as their algebraic input \cite{BorsatoWulff2019,HoareLacroix2020,Hoare2022}. The Nappi-Witten model and its higher-dimensional oscillator analogues provide non-semisimple targets for which such deformations have been analysed directly \cite{KyonoYoshida2016,DemulderEtAl2021,IdiabVanTongeren2024}. Theorem C shows that, on any rotational Medina-Revoy target, the algebraic search for triangular tensors on $\Rone$ may be restricted from the outset to those with zero $h\wedge V$ component, at every frequency pattern. This is a support-level algebraic constraint, not a claim about integrability, conformal invariance, or equivalence of the residual deformations, which continue to require model-specific analysis.

\subsection{Organisation}\label{sec:organization}

Section~\ref{sec:construction} introduces the rotational datum, the block-rank-one loci, marked slices, metrics, and equivariance under the split automorphism group. Section~\ref{sec:structure} proves Theorem A and identifies the moment-map factorisation. Section~\ref{sec:monomial} develops the cancellation criterion, the exponent lattice, and the Smith-normal-form calculus, and proves Theorem B and the constant-rank stratification. Section~\ref{sec:obstruction} proves Theorem C by direct support tracking. Section~\ref{sec:twoplane} treats the complete two-plane trichotomy, gives a closed-form real angular fibre, and compares with the oscillator classification. Section~\ref{sec:conclusion} isolates the two structural obstacles to higher-rank extensions.

\section{Rotational double extensions and marked slices}\label{sec:construction}

\subsection{Conventions}\label{sec:conventions}

All vector spaces and Lie algebras are finite-dimensional over $\mathbb R$; complexifications carry the subscript $(\cdot)_{\mathbb C}$. The Schouten-Nijenhuis bracket on $\Lambda^{\bullet}\g$ is the unique biderivation extending the Lie bracket; it is symmetric on $\Lambda^{2}\g$. Inner products are written $\langle\cdot,\cdot\rangle$, with induced norms $\|\cdot\|$ on all exterior powers. We fix $\mathbb T=\mathbb R/2\pi\mathbb Z$ and, on each oriented Euclidean two-plane $P_{\ell}$, an oriented orthonormal basis $(e_{\ell},f_{\ell})=(e_{\ell},J_{\ell}e_{\ell})$ with area bivector $\Omega_{\ell}=e_{\ell}\wedge f_{\ell}$.

\subsection{The rotational quadratic double extension}\label{sec:doubleext}

\begin{definition}[Rotational datum]\label{def:datum}
A \emph{rotational datum} is a quadruple

\[
\mathcal{D}=(\fb,V,\{P_{\ell}\}_{\ell=1}^{m},\{\lambda_{\ell}\}_{\ell=1}^{m}),
\]
where $\fb$ is a real abelian Lie algebra, $V=\bigoplus_{\ell=1}^{m}P_{\ell}$ is an orthogonal direct sum of oriented Euclidean two-planes, and $\lambda_{\ell}\in\fb^{*}\setminus\{0\}$ are nonzero weights such that
\begin{equation}\label{eq:rotation}
\rho(b)\big|_{P_{\ell}}=\lambda_{\ell}(b)\,J_{\ell},\qquad b\in\fb,
\end{equation}
with $J_{\ell}$ the positive quarter-turn on $P_{\ell}$.
\end{definition}

Set $\fz=\fb^{*}$, viewed as a central summand, and define $\beta\colon\Lambda^{2}V\to\fz$ by
\begin{equation}\label{eq:beta}
\beta(u,v)(b)=\langle\rho(b)u,v\rangle_{V},\qquad u,v\in V,\ b\in\fb.
\end{equation}
On $\g=\fb\oplus V\oplus\fz$ define the bracket
\begin{equation}\label{eq:bracket}
[b,b']=0,\qquad [b,u]=\rho(b)u,\qquad [u,v]=\beta(u,v),\qquad [\fz,\g]=0.
\end{equation}

\begin{proposition}\label{prop:jacobi}
The bracket \eqref{eq:bracket} satisfies the Jacobi identity, and the symmetric form
\begin{equation}\label{eq:Bform}
B(b+u+\xi,\,b'+v+\eta)=\xi(b')+\eta(b)+\langle u,v\rangle_{V}
\end{equation}
is nondegenerate and $\mathrm{ad}$-invariant, so $(\g,B)$ is a quadratic Lie algebra in the sense of Medina-Revoy \cite{MedinaRevoy1985,Bordemann1997}. The quotient $\overline\g=\g/\fz\cong\fb\ltimes_{\rho}V$ is a flat metric Lie algebra in the sense of Milnor \cite{Milnor1976}.
\end{proposition}

\begin{proof}
Brackets involving $\fz$ vanish by centrality; the $VVV$-Jacobi identity is trivial because $[V,V]\subset\fz$; and for $b\in\fb$, $u,v\in V$, the $\fb VV$-Jacobi identity reduces to commutativity of the $\rho(b)$ operators and their skew-adjointness, both immediate from \eqref{eq:rotation}. Invariance of $B$ follows from

\[
B([b,u],v)=\langle\rho(b)u,v\rangle=\beta(u,v)(b)=B(b,[u,v]).\qedhere
\]
\end{proof}

\subsection{Block-rank-one loci and marked slices}\label{sec:locus-def}

For $p\in\fb\wedge V$, write $p=\sum_{\ell=1}^{m}p_{\ell}$ according to $\fb\wedge V=\bigoplus_{\ell}(\fb\wedge P_{\ell})$ and use the canonical identification $\fb\wedge P_{\ell}\cong\fb\otimes P_{\ell}$.

\begin{definition}[Rank-one, null-marker, and positive loci]\label{def:locus}
The \emph{block-rank-one locus} is the semialgebraic subset

\[
\Rone=\bigl\{p\in\fb\wedge V:\rk(p_{\ell})\le 1\ \forall\ell\bigr\}.
\]
For every nonzero block $p_{\ell}$, write $p_{\ell}=h_{\ell}\wedge u_{\ell}$; the line $\mathbb Rh_{\ell}\subset\fb$ is independent of the factorisation and is called the \emph{marker line}. Define

\[
\begin{aligned}
\Rzero&=\bigl\{p\in\Rone:\mathbb Rh_{\ell}\subset\ker\lambda_{\ell}\ \forall\ p_{\ell}\ne 0\bigr\},\\
\Rplus&=\bigl\{p\in\Rone:\mathbb Rh_{\ell}\not\subset\ker\lambda_{\ell}\ \forall\ p_{\ell}\ne 0\bigr\}.
\end{aligned}
\]
The adjective ``positive'' records the canonical orientation of every non-null marker by $\lambda_{\ell}(h_{\ell})>0$.
\end{definition}

Each block of an element of $\Rplus$ admits a factorisation with $\lambda_{\ell}(h_{\ell})>0$, so the positive ray $\mathbb R_{>0}h_{\ell}$ is intrinsic. The locus $\Rone$ also contains mixed points, some of whose markers are null; Theorem~\ref{thm:obstruction} shows that triangularity purges all non-null blocks.

\begin{definition}[Positive marking and marked slice]\label{def:marking}
A \emph{positive marking} is a tuple $H=(h_{1},\dots,h_{m})\in\fb^{m}$ with $\lambda_{\ell}(h_{\ell})>0$ for every $\ell$. Its \emph{marked slice} is
\begin{equation}\label{eq:slice2}
S_{H}=\bigoplus_{\ell=1}^{m}h_{\ell}\wedge P_{\ell}
=\Bigl\{\sum_{\ell}h_{\ell}\wedge u_{\ell}:u_{\ell}\in P_{\ell}\Bigr\}\subset\Lambda^{2}\g.
\end{equation}
The \emph{marking space} is

\[
\mathcal P=\bigl\{H\in\fb^{m}:\lambda_{\ell}(h_{\ell})>0\ \forall\ell\bigr\}.
\]
\end{definition}

The family $\{S_{H}\}_{H\in\mathcal P}$ is a semialgebraic linear cover of $\Rplus$: on the maximal support open set two markings determine the same slice exactly when their marker lines coincide plane by plane, while along boundary faces where some block vanishes distinct slices intersect. This cover is compatible with the fixed-support decomposition of $\Rplus$ but is not a stratification in the topological sense.

\begin{remark}[Dimension of the cover]\label{rem:size}
Let $d=\dim\fb$. In each block $\fb\otimes P_{\ell}$, the rank-at-most-one determinantal cone has dimension $d+1$: its nonzero part is parametrised by pairs $(\fb\setminus\{0\})\times(P_{\ell}\setminus\{0\})$ modulo a single scalar rescaling. Hence $\Rone$ has maximal dimension $m(d+1)$ and codimension $m(d-1)$ inside $\fb\wedge V$, while a fixed marked slice has dimension $2m$ and codimension $2m(d-1)$. The positivity condition defining $\Rplus$ is relatively open on every full-support subset.
\end{remark}

\begin{remark}[Naturality of marked slices]\label{rem:why}
The subspace $S_{H}$ is the smallest linear, $\mathbb T^{m}$-invariant support class that fixes one marker line in $\fb$ per plane while retaining the full phase circle in every plane. The identity $u_{\ell}\wedge J_{\ell}u_{\ell}=\|u_{\ell}\|^{2}\Omega_{\ell}$ places each planewise self-interaction on its own exterior support $\fb\wedge\Lambda^{2}P_{\ell}$. This disjoint-support property drives both the coercivity theorem and the Yang-Baxter support rigidity theorem. The nondegeneracy condition is sharp: if $h_{k}\ne 0$ and $\lambda_{k}(h_{k})=0$, then $[h_{k}\wedge u_{k},h_{k}\wedge u_{k}]=0$ for every $u_{k}$, and no positive radial lower bound can detect that block.
\end{remark}

For $H\in\mathcal P$ define
\begin{align}
E_{\ell}(H)&=2\lambda_{\ell}(h_{\ell})\,h_{\ell}\wedge\Omega_{\ell},\label{eq:El}\\
D_{H}&=\bigoplus_{\ell=1}^{m}\mathbb R E_{\ell}(H),\label{eq:DH}\\
C_{H}&=\Bigl\{\sum_{\ell}t_{\ell}E_{\ell}(H):t_{\ell}\ge 0\Bigr\},\label{eq:CH}\\
X_{H}&=\bigoplus_{1\le i<j\le m}\fb\wedge P_{i}\wedge P_{j}.\label{eq:XH}
\end{align}
Because $E_{\ell}(H)\in\fb\wedge\Lambda^{2}P_{\ell}$ and these summands are pairwise disjoint, the vectors $E_{\ell}(H)$ are linearly independent. Consequently $C_{H}$ is a closed simplicial cone with $m$ extreme rays.

\subsection{Metrics and the canonical trace form}\label{sec:metrics}

Fix a positive-definite inner product $\langle\cdot,\cdot\rangle_{c}$ on $\fb$; use the dual inner product on $\fz$ and declare $\fb$, $V$, $\fz$ mutually orthogonal. Induced norms on exterior powers are denoted $\|\cdot\|$. This auxiliary metric enters only the numerical constants and their equality loci; the direct-sum decomposition \eqref{eq:decomp-intro}, the simplicial cone $C_{H}$, the fibres, properness, and the support obstruction of Theorem~\ref{thm:obstruction} are independent of its choice.

The rotational datum itself determines a canonical symmetric bilinear form
\begin{equation}\label{eq:trace}
\gamma_{\mathcal{D}}(b,b')=-\tfrac{1}{2}\tr_{V}\bigl(\rho(b)\rho(b')\bigr)
=\sum_{a=1}^{m}\lambda_{a}(b)\lambda_{a}(b').
\end{equation}
When the weights $\{\lambda_{\ell}\}$ span $\fb^{*}$, the form $\gamma_{\mathcal{D}}$ is positive definite and gives an intrinsic metric on $\fb$; all quantitative constants below then admit datum-intrinsic formulations. The invariant form $B$ itself is null on $\fb\times\fb$ and therefore cannot serve as a positive norm on $\fb$.

\subsection{Equivariance and invariance of the loci}\label{sec:equivariance}

Let $\Aut_{\mathrm{spl}}^{+}(\mathcal{D})$ denote the \emph{split
automorphism group} of the datum: its elements are triples
$a=(\varphi,\Psi,\pi)$ with $\varphi\in\mathrm{GL}(\fb)$,
$\pi\in S_{m}$, and $\Psi_{\ell}\colon P_{\ell}\to P_{\pi(\ell)}$
an orientation-preserving orthogonal isomorphism satisfying
$\Psi_{\ell}J_{\ell}=J_{\pi(\ell)}\Psi_{\ell}$ and
$\lambda_{\ell}=\lambda_{\pi(\ell)}\circ\varphi$. The action on markings
is $(a\cdot H)_{\pi(\ell)}=\varphi(h_{\ell})$.

\begin{proposition}[Equivariance]\label{prop:equivariance}
For every $a=(\varphi,\Psi,\pi)\in\Aut^{+}_{\mathrm{spl}}(\mathcal{D})$, the map
\begin{equation}\label{eq:atilde}
\tilde a(b+u+\xi)=\varphi(b)+\Psi(u)+(\varphi^{-1})^{*}\xi
\end{equation}
is an isometric Lie algebra automorphism of $(\g,B)$. Setting $H'=a\cdot H$ one has $\Lambda^{2}\tilde a(S_{H})=S_{H'}$, $\Lambda^{3}\tilde a(E_{\ell}(H))=E_{\pi(\ell)}(H')$, and
\begin{equation}\label{eq:equivariance-eq}
\Sigma_{H'}\circ\Lambda^{2}\tilde a=\Lambda^{3}\tilde a\circ\Sigma_{H},
\end{equation}
componentwise in the radial and angular parts. Moreover, the loci $\Rone$, $\Rzero$, $\Rplus\subset\fb\wedge V$ are invariant under $\Lambda^{2}\tilde a$ for every $a\in\Aut_{\mathrm{spl}}^{+}(\mathcal D)$.
\end{proposition}

\begin{proof}
Weight compatibility gives $\Psi\rho(b)=\rho(\varphi b)\Psi$. For $u,v\in V$ and $b'\in\fb$,

\[
((\varphi^{-1})^{*}\beta(u,v))(b')=\beta(u,v)(\varphi^{-1}b')=\langle\rho(\varphi^{-1}b')u,v\rangle=\langle\rho(b')\Psi u,\Psi v\rangle,
\]
so $\tilde a$ preserves the bracket; it preserves $B$ because $\Psi$ is orthogonal and $((\varphi^{-1})^{*}\xi)(\varphi b)=\xi(b)$. Preservation of orientation makes the identities on $S_H$ and $E_\ell(H)$ direct; \eqref{eq:equivariance-eq} is naturality of the Schouten bracket. Finally, $\Lambda^{2}\tilde a$ maps $\fb\wedge P_{\ell}$ onto $\fb\wedge P_{\pi(\ell)}$ preserving tensor rank, and preserves the marker-line condition because $\lambda_{\pi(\ell)}\circ\varphi=\lambda_{\ell}$; hence each of $\Rone$, $\Rzero$, $\Rplus$ is invariant.
\end{proof}

\begin{remark}[Datum-relative naturality]\label{rem:intrinsic}
The loci and slices of Definition~\ref{def:locus} are intrinsic relative to the rotational datum $\mathcal D$. When weights repeat, the full automorphism group of $(\g,B)$ may act on the isotypic component through orthogonal transformations that mix isomorphic planes without preserving the chosen refinement $V=\bigoplus P_{\ell}$; such transformations lie outside $\Aut_{\mathrm{spl}}^{+}(\mathcal D)$ and need not preserve block rank in the chosen splitting. The datum itself is part of the input, and every intrinsic statement below is understood relative to it.
\end{remark}

\section{Radial toric geometry and coercivity}\label{sec:structure}

The organising identity is planar: for $u\in P_{\ell}$, expanding $u=ae_{\ell}+bf_{\ell}$ gives $u\wedge J_{\ell}u=(a^{2}+b^{2})e_{\ell}\wedge f_{\ell}=\|u\|^{2}\Omega_{\ell}$. Each planewise self-interaction is then confined to a fixed ray, and summing over planes builds the simplicial cone $C_{H}$.

\subsection{Schouten formulae}

For $x,y,z,w\in\g$ the Schouten-Nijenhuis bracket satisfies
\begin{align}
[x\wedge y,\,z\wedge w]
&=[x,z]\wedge y\wedge w-[x,w]\wedge y\wedge z\notag\\
&\quad -[y,z]\wedge x\wedge w+[y,w]\wedge x\wedge z,\label{eq:schouten1}\\
[x\wedge y,\,x\wedge y]&=2[x,y]\wedge x\wedge y.\label{eq:schouten2}
\end{align}

\begin{lemma}[Planewise Schouten identities]\label{lem:planewise}
For $u_{\ell}\in P_{\ell}$,
\begin{equation}\label{eq:planewise1}
[h_{\ell}\wedge u_{\ell},\,h_{\ell}\wedge u_{\ell}]=\|u_{\ell}\|^{2}E_{\ell}(H),
\end{equation}
and, for $i\ne j$,
\begin{equation}\label{eq:planewise2}
[h_{i}\wedge u_{i},\,h_{j}\wedge u_{j}]
=\lambda_{j}(h_{i})\,h_{j}\wedge u_{i}\wedge J_{j}u_{j}
-\lambda_{i}(h_{j})\,h_{i}\wedge J_{i}u_{i}\wedge u_{j}.
\end{equation}
\end{lemma}

\begin{proof}
By \eqref{eq:schouten2} and the planar identity,

\[
\begin{aligned}
[h_{\ell}\wedge u_{\ell},h_{\ell}\wedge u_{\ell}]
&=2\lambda_{\ell}(h_{\ell})\,J_{\ell}u_{\ell}\wedge h_{\ell}\wedge u_{\ell}\\
&=2\lambda_{\ell}(h_{\ell})\,h_{\ell}\wedge (u_{\ell}\wedge J_{\ell}u_{\ell})
=\|u_{\ell}\|^{2}E_{\ell}(H),
\end{aligned}
\]
which is \eqref{eq:planewise1}. Applying \eqref{eq:schouten1} to $(h_{i},u_{i},h_{j},u_{j})$ with $[h_{i},h_{j}]=0$ and $[u_{i},u_{j}]\in\fz$ absent from $\Lambda^{3}(\fb\oplus V)$ yields \eqref{eq:planewise2}.
\end{proof}

\subsection{The structure theorem}

\begin{theorem}[Radial toric structure, coercivity, and properness]\label{thm:structure}
Let $H\in\mathcal P$ and $r=\sum_{\ell}h_{\ell}\wedge u_{\ell}\in S_{H}$.
\begin{enumerate}
\item[\rm (a)] Under $\Lambda^{3}\g\supset D_{H}\oplus X_{H}$,
\begin{equation}\label{eq:decomp}
[r,r]=\Phi_{H}(r)+\mu_{H}(r),
\end{equation}
with
\begin{align}
\Phi_{H}(r)&=\sum_{\ell=1}^{m}\|u_{\ell}\|^{2}E_{\ell}(H),\label{eq:PhiH}\\
\mu_{H}(r)&=2\sum_{1\le i<j\le m}\bigl(\lambda_{j}(h_{i})\,h_{j}\wedge u_{i}\wedge J_{j}u_{j}
-\lambda_{i}(h_{j})\,h_{i}\wedge J_{i}u_{i}\wedge u_{j}\bigr).\label{eq:muH}
\end{align}
Moreover, $\Phi_{H}(S_{H})=C_{H}$, and for $Y=\sum_{\ell}t_{\ell}E_{\ell}(H)\in C_{H}$ the radial fibre is
\begin{equation}\label{eq:TY}
T_{Y}\coloneqq\Phi_{H}^{-1}(Y)=\prod_{t_{\ell}>0}S^{1}_{\sqrt{t_{\ell}}}\times\prod_{t_{\ell}=0}\{0\}.
\end{equation}
Under the identification $S_{H}\cong\bigoplus_{\ell}P_{\ell}$ with $\mathbb T^{m}$-action by independent rotations and symplectic form $\omega(v,w)=\sum_{\ell}\langle v_{\ell},J_{\ell}w_{\ell}\rangle$,
\begin{equation}\label{eq:moment-factor}
\Phi_{H}=\iota_{H}\circ\mu_{\mathbb T},\quad
\mu_{\mathbb T}(u)=\tfrac{1}{2}\sum_{\ell}\|u_{\ell}\|^{2}\varepsilon_{\ell}^{*},\quad
\iota_{H}(\varepsilon_{\ell}^{*})=2E_{\ell}(H),
\end{equation}
with $\iota_{H}$ an injective linear map.

\item[\rm (b)] Set
\begin{equation}\label{eq:kappa0}
\kappa_{0}(H)=2\left(\sum_{\ell=1}^{m}\frac{\|h_{\ell}\|^{2}}{\lambda_{\ell}(h_{\ell})^{2}}\right)^{-1/2}.
\end{equation}
Then
\begin{equation}\label{eq:coercive}
\|\Phi_{H}(r)\|\ge \kappa_{0}(H)\,\|r\|^{2},
\end{equation}
and the constant is optimal: for $\|r\|^{2}=R^{2}>0$, equality holds iff
\begin{equation}\label{eq:equality-locus}
\|u_{\ell}\|^{2}=\frac{R^{2}\lambda_{\ell}(h_{\ell})^{-2}}{\sum_{j=1}^{m}\|h_{j}\|^{2}\lambda_{j}(h_{j})^{-2}}\qquad(1\le\ell\le m).
\end{equation}

\item[\rm (c)] The restricted Schouten map

\[
\Sigma_{H}\colon S_{H}\to D_{H}\oplus X_{H},\quad \Sigma_{H}(r)=[r,r],
\]
is proper. Its image over $Y\in C_{H}$ is $Y+F_{Y}(T_{Y})$, where $F_{Y}=\mu_{H}|_{T_{Y}}$, and $F_{Y}(T_{Y})$ is compact.
\end{enumerate}
\end{theorem}

\begin{proof}
Set $r_{\ell}=h_{\ell}\wedge u_{\ell}$. Symmetry of the Schouten bracket gives $[r,r]=\sum_{\ell}[r_{\ell},r_{\ell}]+2\sum_{i<j}[r_{i},r_{j}]$; Lemma~\ref{lem:planewise} then yields \eqref{eq:PhiH}-\eqref{eq:muH}. The summands $\fb\wedge\Lambda^{2}P_{\ell}$ and $\fb\wedge P_{i}\wedge P_{j}$ ($i\ne j$) are pairwise disjoint in $\Lambda^{3}\g$, giving \eqref{eq:decomp}. Since the squared radii $\|u_{\ell}\|^{2}$ vary independently in $[0,\infty)$ and the $E_{\ell}(H)$ are linearly independent, $\Phi_{H}(S_{H})=C_{H}$ and the fibres are as in \eqref{eq:TY}. The identification \eqref{eq:moment-factor} of $\Phi_{H}$ with the standard moment map followed by a linear rescaling is immediate.

For (b), set $t_{\ell}=\|u_{\ell}\|^{2}$. Orthogonality of the summands gives

\[
\|\Phi_{H}(r)\|^{2}=4\sum_{\ell}\lambda_{\ell}(h_{\ell})^{2}\|h_{\ell}\|^{2}t_{\ell}^{2},\qquad
\|r\|^{2}=\sum_{\ell}\|h_{\ell}\|^{2}t_{\ell}.
\]
Apply Cauchy-Schwarz with $a_{\ell}=\|h_{\ell}\|/\lambda_{\ell}(h_{\ell})$ and $b_{\ell}=2\lambda_{\ell}(h_{\ell})\|h_{\ell}\|t_{\ell}$:

\[
2\|r\|^{2}=\sum_{\ell}a_{\ell}b_{\ell}\le\Bigl(\sum_{\ell}a_{\ell}^{2}\Bigr)^{1/2}\Bigl(\sum_{\ell}b_{\ell}^{2}\Bigr)^{1/2}
=\Bigl(\sum_{\ell}\frac{\|h_{\ell}\|^{2}}{\lambda_{\ell}(h_{\ell})^{2}}\Bigr)^{1/2}\|\Phi_{H}(r)\|,
\]
which is \eqref{eq:coercive}. Equality forces $t_{\ell}\propto\lambda_{\ell}(h_{\ell})^{-2}$; imposing $\|r\|^{2}=R^{2}$ gives \eqref{eq:equality-locus}. Every coordinate in \eqref{eq:equality-locus} is strictly positive, so the equality locus meets the interior of the radius simplex and $\kappa_{0}(H)$ is optimal.

For (c), since $D_{H}\perp X_{H}$, estimate \eqref{eq:coercive} gives $\|\Sigma_{H}(r)\|\ge\kappa_{0}(H)\|r\|^{2}$, whence $\Sigma_{H}$ is proper. The fibre formula follows by fixing the $D_{H}$-component of $[r,r]$, and $F_{Y}(T_{Y})$ is compact because $T_{Y}$ is compact and $\mu_{H}$ is continuous.
\end{proof}

\begin{remark}[Moment-map interpretation and metric invariance]\label{rem:moment}
The radial map \eqref{eq:moment-factor} is a standard toric moment map followed by an injective linear rescaling; thus $C_{H}$ is the moment cone. Its relative interior is the regular-value locus, and its proper faces are indexed by the vanishing plane components. The properness of $\Sigma_{H}$ and Corollary~\ref{cor:image} below are independent of the auxiliary Euclidean norm on $\fb$; only the numerical value of $\kappa_{0}(H)$ depends on it. Under a change of positive-definite metric $g\mapsto g'$ on $\fb$ with $g'=A^{*}g$ for $A\in\mathrm{GL}(\fb)$ symmetric-positive, the ratio $\kappa_{0}(H)/\kappa_{0}'(H)$ is a rational function of the marker components and depends only on $A$, not on $r$. When the weights $\{\lambda_{\ell}\}$ span $\fb^{*}$, the trace form $\gamma_{\mathcal{D}}$ from \eqref{eq:trace} is positive-definite and yields the intrinsic constant
\begin{equation}\label{eq:kappagamma}
\kappa_{\gamma}(H)=2\Bigl(\sum_{\ell=1}^{m}\frac{\sum_{a=1}^{m}\lambda_{a}(h_{\ell})^{2}}{\lambda_{\ell}(h_{\ell})^{2}}\Bigr)^{-1/2}.
\end{equation}
\end{remark}

\begin{example}[Equality in the coercive estimate]\label{ex:equality}
Let $\fb=\mathbb Rb_{1}\oplus\mathbb Rb_{2}$ with the standard Euclidean metric, $\lambda_{1}=b_{1}^{*}$, $\lambda_{2}=b_{2}^{*}$, and $H=(b_{1},b_{2})$. Then $\kappa_{0}(H)=\sqrt{2}$. On the sphere $\|r\|^{2}=R^{2}$, equality in Theorem~\ref{thm:structure}(b) occurs precisely at

\[
\|u_{1}\|^{2}=\|u_{2}\|^{2}=R^{2}/2.
\]
Since $\|E_{1}(H)\|=\|E_{2}(H)\|=2$, these radii yield $\|\Phi_{H}(r)\|=\sqrt{2}R^{2}=\kappa_{0}(H)R^{2}$: the optimal point is the barycentre of the radius simplex.
\end{example}

\begin{corollary}[Image characterisation and affine fibres]\label{cor:image}
The image $\Sigma_{H}(S_{H})$ is closed and semialgebraic. For $\Theta\in\Lambda^{3}\g$, the equation $[r,r]=\Theta$ with $r\in S_{H}$ has a solution iff
\begin{enumerate}
\item[\rm (i)] $\Theta\in D_{H}\oplus X_{H}$;
\item[\rm (ii)] $Y\coloneqq\pr_{D_{H}}\Theta\in C_{H}$;
\item[\rm (iii)] $\pr_{X_{H}}\Theta\in F_{Y}(T_{Y})$.
\end{enumerate}
The solution set, when non-empty, is a closed semialgebraic subset of the compact torus $T_{Y}$.
\end{corollary}

\begin{proof}
A proper continuous map between Euclidean spaces is closed, and semialgebraicity of the image follows from Tarski-Seidenberg \cite{BochnakCosteRoy1998}. Conditions (i)--(iii) are Theorem~\ref{thm:structure}(a,c); the solution set is a closed semialgebraic subset of a compact torus by continuity.
\end{proof}

\section{Monomial angular fibres and rank stratification}\label{sec:monomial}

Fixing $Y\in C_{H}$ freezes the moduli $\|u_{\ell}\|$ and leaves purely angular data. Complexifying each plane and diagonalising the $J_{\ell}$ turns the angular map into a Laurent monomial map. Two objects control the picture: the \emph{support graph} $\Gamma_{H}$, recording which pairs of planes interact via \eqref{eq:planewise2}, and the \emph{signed exponent configuration} $A_{H,Y}$, recording which eigenspace components survive coefficient cancellation.

\subsection{Support graph, eigenspace decomposition, and cancellation}

For $i<j$, declare $\{i,j\}$ an edge of $\Gamma_{H}$ when
\begin{equation}\label{eq:edge}
\lambda_{j}(h_{i})\ne 0\quad\text{or}\quad\lambda_{i}(h_{j})\ne 0.
\end{equation}
Complexify $P_{\ell}$ as $P_{\ell,\mathbb C}=L_{\ell}^{+}\oplus L_{\ell}^{-}$, where $J_{\ell}|_{L_{\ell}^{\varepsilon}}=i\varepsilon$ for $\varepsilon\in\{\pm 1\}$; let $z_{\ell}^{\varepsilon}$ denote the coordinate on $L_{\ell}^{\varepsilon}$. By \eqref{eq:planewise2}, the component of $\mu_{H}$ in $\fb_{\mathbb C}\wedge L_{i}^{\varepsilon}\wedge L_{j}^{\delta}$ is
\begin{equation}\label{eq:coeff}
c_{ij}^{\varepsilon\delta}(H)\,z_{i}^{\varepsilon}z_{j}^{\delta},
\qquad c_{ij}^{\varepsilon\delta}(H)=2i\bigl(\delta\lambda_{j}(h_{i})\,h_{j}-\varepsilon\lambda_{i}(h_{j})\,h_{i}\bigr).
\end{equation}

\begin{lemma}[Cancellation criterion]\label{lem:cancellation}
Let $\{i,j\}$ be an edge of $\Gamma_{H}$ and set $\hh_{\ell}=h_{\ell}/\lambda_{\ell}(h_{\ell})$. Then
\begin{equation}\label{eq:cancel}
c_{ij}^{\varepsilon\delta}(H)=0\iff \hh_{j}=(\delta/\varepsilon)\hh_{i}.
\end{equation}
An interacting pair either has all four signed coefficients nonzero or loses exactly the opposite pair indexed by $(\varepsilon,\delta)$ and $(-\varepsilon,-\delta)$. If $\{i,j\}$ is not an edge, all four coefficients vanish.
\end{lemma}

\begin{proof}
If $\{i,j\}$ is not an edge, both $\lambda_{j}(h_{i})$ and $\lambda_{i}(h_{j})$ vanish, hence so do all coefficients. Assume $\{i,j\}$ is an edge and $c_{ij}^{\varepsilon\delta}(H)=0$, i.e.
\begin{equation}\label{eq:cancel-eq}
\delta\lambda_{j}(h_{i})\,h_{j}=\varepsilon\lambda_{i}(h_{j})\,h_{i}.
\end{equation}
Neither scalar vanishes alone: if $\lambda_{j}(h_{i})=0$ and $\lambda_{i}(h_{j})\ne 0$, then $h_{i}=0$, contradicting $\lambda_{i}(h_{i})>0$. Applying $\lambda_{i}$ to \eqref{eq:cancel-eq} gives $\delta\lambda_{j}(h_{i})\lambda_{i}(h_{j})=\varepsilon\lambda_{i}(h_{j})\lambda_{i}(h_{i})$, whence $\lambda_{j}(h_{i})=(\varepsilon/\delta)\lambda_{i}(h_{i})$; applying $\lambda_{j}$ gives $\lambda_{i}(h_{j})=(\delta/\varepsilon)\lambda_{j}(h_{j})$. Substituting back returns $\hh_{j}=(\delta/\varepsilon)\hh_{i}$. The converse is immediate, and $c_{ij}^{-\varepsilon,-\delta}(H)=-c_{ij}^{\varepsilon\delta}(H)$ makes cancellations occur in opposite pairs.
\end{proof}

Put $d=\dim\fb$ and $r_{ij}=\dim_{\mathbb R}\spanR\{\lambda_{i},\lambda_{j}\}$. For $\tau\in\{\pm 1\}$, set

\[
\mathcal C_{ij}^{\tau}=\{H\in\mathcal P:\hh_{j}=\tau\hh_{i}\}.
\]

\begin{proposition}[Cancellation strata]\label{prop:strata}
The set $\mathcal C_{ij}^{\tau}$ is nonempty iff $(1,\tau)$ lies in the image of $(\lambda_{i},\lambda_{j})\colon\fb\to\mathbb R^{2}$. If nonempty, it is a smooth semialgebraic submanifold of $\mathcal P$ of codimension $d+r_{ij}-2$, parametrised by
\begin{equation}\label{eq:strata-param}
\{p\in\fb:\lambda_{i}(p)=1,\ \lambda_{j}(p)=\tau\}\times\mathbb R_{>0}^{2}\times\prod_{k\ne i,j}\{h_{k}:\lambda_{k}(h_{k})>0\}.
\end{equation}
If $r_{ij}=2$ both signs $\tau$ occur; if $r_{ij}=1$ nonemptiness is equivalent to $\lambda_{j}=\tau\lambda_{i}$, and then only that sign occurs. Consequently, if $d\ge 2$ non-cancellation is open, dense, and of full measure in $\mathcal P$; if $d=1$ then $\Gamma_{H}$ is complete and non-cancellation holds precisely when $\lambda_{i}\ne\pm\lambda_{j}$ for all $i\ne j$.
\end{proposition}

\begin{proof}
Setting $p=\hh_{i}$ gives $\lambda_{i}(p)=1$ and $\lambda_{j}(p)=\tau$; the affine subspace of admissible $p$ has dimension $d-r_{ij}$. Adjoining the two positive scalars $a,b>0$ with $h_{i}=ap$, $h_{j}=b\tau p$, and the free coordinates $h_{k}$ for $k\ne i,j$, gives $\dim\mathcal C_{ij}^{\tau}=md-d-r_{ij}+2$, whence the codimension. For $d=1$ all nonzero weights are proportional, so $\Gamma_{H}$ is complete; writing $h_{i}=a_{i}b$ for a generator $b$, $\hh_{i}=b/\lambda_{i}(b)$ and \eqref{eq:cancel} reduces to $\lambda_{i}=\pm\lambda_{j}$.
\end{proof}

\subsection{Exponent lattices and the angular fibre theorem}

Fix an active support $I\subset\{1,\dots,m\}$ and a radial value $Y=\sum_{\ell\in I}t_{\ell}E_{\ell}(H)$ with $t_{\ell}>0$ for every $\ell\in I$. Write
\begin{equation}\label{eq:uell}
u_{\ell}=\sqrt{t_{\ell}}\bigl(\cos\theta_{\ell}\,e_{\ell}+\sin\theta_{\ell}\,f_{\ell}\bigr),\quad s_{\ell}=e^{i\theta_{\ell}};
\end{equation}
in the eigenvector basis $v_{\ell}^{\pm}=(e_{\ell}\mp if_{\ell})/\sqrt{2}$ the coordinates become $w_{\ell}^{\pm}=\sqrt{t_{\ell}/2}\,s_{\ell}^{\pm 1}$. Define the \emph{signed exponent configuration}
\begin{equation}\label{eq:AHY}
A_{H,Y}=\bigl\{\varepsilon e_{i}+\delta e_{j}\in\mathbb Z^{I}:i<j,\ i,j\in I,\ c_{ij}^{\varepsilon\delta}(H)\ne 0\bigr\},
\end{equation}
the \emph{exponent lattice} $L_{A}=\spanZ(A_{H,Y})\subset\mathbb Z^{I}$, and the \emph{character kernel}
\begin{equation}\label{eq:KA}
K_{A}=\bigl\{\theta\in\mathbb T^{I}:e^{i\langle\alpha,\theta\rangle}=1\ \forall\,\alpha\in A_{H,Y}\bigr\}.
\end{equation}
When $A_{H,Y}=\varnothing$ set $\rk L_{A}=0$, $K_{A}=\mathbb T^{I}$, and $\conv(A_{H,Y})=\varnothing$. For a Laurent polynomial map $F=(F_{\nu})$ with monomial coordinates, its \emph{joint Newton polytope} $\Newt(F)$ is the convex hull of the union of the exponent supports of its coordinate functions; in our setting this coincides with $\conv(A_{H,Y})$.

\begin{theorem}[Angular fibre theorem]\label{thm:monomial}
After complexification and projection to joint eigenspaces, the angular map $F_{Y}=\mu_{H}|_{T_{Y}}$ takes the Laurent monomial form
\begin{equation}\label{eq:Laurent}
(s_{\ell})_{\ell\in I}\longmapsto\bigl(a_{\alpha}(H,Y)\,s^{\alpha}\bigr)_{\alpha\in A_{H,Y}},
\end{equation}
each $a_{\alpha}(H,Y)\ne 0$ lying in a coefficient summand $\fb_{\mathbb C}\wedge L_{i}^{\varepsilon}\wedge L_{j}^{\delta}$ distinct from the others. Then:
\begin{enumerate}
\item[\rm (a)] At every point of $T_{Y}$, $\rk_{\mathbb R}dF_{Y}=\dim_{\mathbb Q}\spanQ A_{H,Y}=\rk L_{A}$.
\item[\rm (b)] The fibres of $F_{Y}$ are cosets of $K_{A}$, and $F_{Y}$ factors through a smooth embedding $\mathbb T^{I}/K_{A}\hookrightarrow X_{H}\otimes_{\mathbb R}\mathbb C$ whose image is a translate of a compact torus of dimension $\rk L_{A}$. The component group of $K_{A}$ is Pontryagin dual to $L_{A}^{\mathrm{sat}}/L_{A}$, where $L_{A}^{\mathrm{sat}}=(\spanQ A_{H,Y})\cap\mathbb Z^{I}$.
\item[\rm (c)] The joint exponent polytope of the coordinate collection is $\Newt(F_{Y})=\conv(A_{H,Y})$.
\end{enumerate}
\end{theorem}

\begin{proof}
Formula \eqref{eq:Laurent} follows from \eqref{eq:coeff}--\eqref{eq:uell} with the radii absorbed into the coefficients; distinct signed exponents correspond to distinct summands $\fb_{\mathbb C}\wedge L_{i}^{\varepsilon}\wedge L_{j}^{\delta}$ of $\Lambda^{3}\g_{\mathbb C}$, so their coefficient contributions are linearly independent.

For (a), the $\alpha$-component of $dF_{Y}$ is $i\langle\alpha,v\rangle a_{\alpha}e^{i\langle\alpha,\theta\rangle}$; independence of summands gives $\ker dF_{Y}=A_{H,Y}^{\perp}$, whence the rank formula.

For (b), $F_{Y}(\theta)=F_{Y}(\theta')$ iff $s^{\alpha}=(s')^{\alpha}$ for every $\alpha\in A_{H,Y}$, iff $\theta-\theta'\in K_{A}$. The induced map on the quotient is injective; compactness of the domain and Hausdorffness of the target make it an embedding onto a translate of a compact torus of dimension $\rk L_{A}$. The component-group statement is standard Pontryagin duality for kernels of homomorphisms of compact tori.

Part (c) is immediate from \eqref{eq:Laurent}.
\end{proof}

\begin{proposition}[Rank on fixed-support strata]\label{prop:rank-strata}
For $I\subset\{1,\dots,m\}$, let $S_{H}(I)=\{\sum_{i\in I}h_{i}\wedge u_{i}:u_{i}\in P_{i}\setminus\{0\}\}$. For $r\in S_{H}(I)$, $Y=\Phi_{H}(r)$, and $L_{A}$ the exponent lattice of $A_{H,Y}$,
\begin{equation}\label{eq:rank-strata}
\rk\, d(\Sigma_{H}|_{S_{H}(I)})_{r}=|I|+\rk L_{A}.
\end{equation}
In particular, $\Sigma_{H}$ has constant rank on every fixed-support subset of $\Rplus$.
\end{proposition}

\begin{proof}
At $r$,

\[
T_{r}S_{H}(I)=\bigoplus_{i\in I}h_{i}\wedge P_{i}.
\]
For $v=(v_{i})_{i\in I}$, $d\Phi_{H}(v)=2\sum_{i\in I}\langle u_{i},v_{i}\rangle E_{i}(H)$: this map has rank $|I|$ with kernel $\bigoplus_{i\in I}\mathbb R J_{i}u_{i}=T_{r}T_{Y}$. On this kernel, $d\Sigma_{H}=d\mu_{H}=dF_{Y}$ has rank $\rk L_{A}$ by Theorem~\ref{thm:monomial}(a). Because the two differentials take values in the orthogonal summands $D_{H}$ and $X_{H}$, their rank contributions add.
\end{proof}

\begin{remark}[Boundary behaviour and rank semicontinuity]\label{rem:boundary}
The convention $A_{H,Y}=\varnothing\Rightarrow\rk L_{A}=0$ makes Proposition~\ref{prop:rank-strata} unconditional: if the active planes do not interact, the differential rank is the radial rank $|I|$. If $J\subset I$ arises by letting $t_{\ell}\to 0$ for $\ell\in I\setminus J$, then under the natural embedding $\mathbb Z^{J}\hookrightarrow\mathbb Z^{I}$, $A_{H,Y_{J}}\subset A_{H,Y_{I}}$, so

\[
|J|+\rk L_{A(J)}\le |I|+\rk L_{A(I)}.
\]
The differential rank can only drop on the boundary of the radial cone. The isotropy may acquire additional circle factors, and its finite component group may change; these discrete jumps are exactly recorded by the Smith normal form on each fixed-support subset.
\end{remark}

\subsection{Defining binomial equations and Smith normal form}

\begin{proposition}[Defining binomial ideal]\label{prop:binomial}
Let $A=A_{H,Y}$ with chosen nonzero coefficients $a_{\alpha}\in\mathbb C^{*}$, and set $\Lambda_{A}=\{c\in\mathbb Z^{A}:\sum_{\alpha}c_{\alpha}\alpha=0\}$. In scalar eigencoordinates, the Zariski closure in $(\mathbb C^{*})^{A}$ of the compact angular image $F_{Y}(T_{Y})$ is the translated subtorus
\begin{equation}\label{eq:XA}
X_{A}=\bigl\{y\in(\mathbb C^{*})^{A}:y^{c}=a^{c}\ \forall\,c\in\Lambda_{A}\bigr\},
\end{equation}
with $y^{c}=\prod_{\alpha}y_{\alpha}^{c_{\alpha}}$ and $a^{c}=\prod_{\alpha}a_{\alpha}^{c_{\alpha}}$. Equivalently, $X_{A}$ is cut out by the Laurent binomials

\[
\prod_{c_{\alpha}>0}y_{\alpha}^{c_{\alpha}}-\Bigl(\prod_{\alpha}a_{\alpha}^{c_{\alpha}}\Bigr)\prod_{c_{\alpha}<0}y_{\alpha}^{-c_{\alpha}},\quad c\in\Lambda_{A}.
\]
\end{proposition}

\begin{proof}
The real torus $\mathbb T^{I}$ is Zariski-dense in $(\mathbb C^{*})^{I}$, so the closure of $F_{Y}(T_{Y})$ equals the image of the complexified monomial map. This map is the translate by $(a_{\alpha})$ of the algebraic-torus homomorphism induced by $e_{\alpha}^{*}\mapsto\alpha$. The image of a homomorphism of algebraic tori is a closed subtorus, and a character $y^{c}$ is trivial on the untranslated image iff $\sum c_{\alpha}\alpha=0$ \cite{EisenbudSturmfels1996,CoxLittleSchenck2011}.
\end{proof}

\begin{remark}[Smith-normal-form algorithm]\label{rem:smith}
Let $M_{A}\in\mathrm{Mat}_{|I|\times|A|}(\mathbb Z)$ have the vectors in $A$ as columns. Choose unimodular matrices $U,V$ such that

\[
UM_{A}V=\diag(d_{1},\dots,d_{r},0,\dots,0),\qquad d_{1}\mid\cdots\mid d_{r}.
\]
Then $r=\rk L_{A}$, $L_{A}^{\mathrm{sat}}/L_{A}\cong\bigoplus_{j=1}^{r}\mathbb Z/d_{j}\mathbb Z$, $K_{A}\cong\mathbb T^{|I|-r}\times\prod_{j=1}^{r}\mu_{d_{j}}$ (the last isomorphism noncanonical), and the last $|A|-r$ columns of $V$ form a $\mathbb Z$-basis of $\ker_{\mathbb Z}M_{A}=\Lambda_{A}$. The corresponding $|A|-r$ binomials of Proposition~\ref{prop:binomial} generate the defining ideal; all other relations follow multiplicatively.
\end{remark}

\begin{corollary}[Graph formula without cancellation]\label{cor:graph}
Assume no cancellation on the active graph $\Gamma_{H}|_{I}$. Then $\rk dF_{Y}$ equals the number of $i\in I$ incident to an active edge; if $I_{1},\dots,I_{c}$ are the nontrivial connected components and $J\subset I$ the isolated vertices, then $K_{A}\cong\mathbb T^{|J|}\times(\mathbb Z/2\mathbb Z)^{c}$, the order-two element on component $I_{a}$ being the simultaneous half-turn $\theta_{i}=\pi$ for $i\in I_{a}$.
\end{corollary}

\begin{proof}
Each active edge $\{i,j\}$ contributes the four vectors $\pm e_{i}\pm e_{j}$ without cancellation. For a connected component $I_{a}$ these span $\mathbb R^{I_{a}}$ and their integral span is $\{x\in\mathbb Z^{I_{a}}:\sum_{i\in I_{a}}x_{i}\equiv 0\pmod 2\}$, whose annihilator in $\mathbb T^{I_{a}}$ is $\{0,\pi(1,\dots,1)\}$. Isolated vertices contribute free circle factors.
\end{proof}

\begin{example}[Three active planes: graph, Smith form, and binomials]\label{ex:threeplane}
Let $\fb=\mathbb Rb_{1}\oplus\mathbb Rb_{2}$ with dual basis $(b_{1}^{*},b_{2}^{*})$, and set

\[
\lambda_{1}=b_{1}^{*},\quad\lambda_{2}=b_{2}^{*},\quad\lambda_{3}=b_{1}^{*}+b_{2}^{*},\quad H=(b_{1},b_{2},b_{1}+b_{2}).
\]
All marking positivities hold. The pair $\{1,2\}$ is not an edge because $\lambda_{2}(h_{1})=\lambda_{1}(h_{2})=0$; the pairs $\{1,3\}$ and $\{2,3\}$ are edges, and no cancellation occurs because $\hh_{1}=b_{1}$, $\hh_{2}=b_{2}$, $\hh_{3}=(b_{1}+b_{2})/2$ are not proportional up to sign. Thus $\Gamma_{H}$ is the path $1$--$3$--$2$.

For a full-support radial value $Y=\sum_{\ell}t_{\ell}E_{\ell}(H)$ with $t_{\ell}>0$, order the exponents as

\[
\alpha_{1}=e_{1}+e_{3},\ \alpha_{2}=e_{1}-e_{3},\ \alpha_{3}=-e_{1}+e_{3},\ \alpha_{4}=-e_{1}-e_{3},
\]

\[
\alpha_{5}=e_{2}+e_{3},\ \alpha_{6}=e_{2}-e_{3},\ \alpha_{7}=-e_{2}+e_{3},\ \alpha_{8}=-e_{2}-e_{3}.
\]
The exponent matrix is
\begin{equation}\label{eq:MA}
M_{A}=\begin{pmatrix}
1 & 1 & -1 & -1 & 0 & 0 & 0 & 0 \\
0 & 0 & 0 & 0 & 1 & 1 & -1 & -1 \\
1 & -1 & 1 & -1 & 1 & -1 & 1 & -1
\end{pmatrix},
\end{equation}
whose Smith normal form is $\diag(1,1,2)\oplus 0_{3\times 5}$. Hence $\rk L_{A}=3$, $L_{A}=\{x\in\mathbb Z^{3}:x_{1}+x_{2}+x_{3}\equiv 0\pmod 2\}$, $L_{A}^{\mathrm{sat}}=\mathbb Z^{3}$, and $K_{A}\cong\mathbb Z/2\mathbb Z$, generated by the simultaneous half-turn $(\pi,\pi,\pi)$; Proposition~\ref{prop:rank-strata} then yields differential rank $3+3=6$ on the full-support subset.

A $\mathbb Z$-basis of $\ker_{\mathbb Z}M_{A}$ is

\[
\alpha_{1}+\alpha_{4}=\alpha_{2}+\alpha_{3}=\alpha_{5}+\alpha_{8}=\alpha_{6}+\alpha_{7}=0,\quad
\alpha_{1}+\alpha_{6}-\alpha_{2}-\alpha_{5}=0,
\]
so the translated subtorus in $(\mathbb C^{*})^{8}$ is cut out by

\[
\begin{aligned}
x_{1}x_{4}&=a_{1}a_{4}, & x_{2}x_{3}&=a_{2}a_{3}, & x_{5}x_{8}&=a_{5}a_{8},\\
x_{6}x_{7}&=a_{6}a_{7}, & x_{1}x_{6}&=\frac{a_{1}a_{6}}{a_{2}a_{5}}\,x_{2}x_{5}. &&
\end{aligned}
\]
The example exercises every output of the Smith-normal-form procedure: rank, a nontrivial component group, and a finite binomial basis.
\end{example}

\section{Rank-one support rigidity for triangular tensors}\label{sec:obstruction}

The radial support of a marked slice cannot be reached by any other exterior component of a general bivector. This support isolation, combined with the planewise identity of Lemma~\ref{lem:planewise}, yields a statement for arbitrary bivectors, subject only to the block-rank-one condition on the $\fb\wedge V$ component. No angular data are needed; the proof is a direct projection onto the radial exterior sectors.

\begin{theorem}[Rank-one support rigidity]\label{thm:obstruction}
Let $r\in\Lambda^{2}\g$ and let

\[
p=\pr_{\fb\wedge V}(r)=\sum_{\ell=1}^{m}p_{\ell}
\]
be its $\fb\wedge V$ component. Assume $p\in\Rone$. If $[r,r]=0$, then $p\in\Rzero$; equivalently, every nonzero factorisation $p_{\ell}=h_{\ell}\wedge u_{\ell}$ satisfies

\[
\lambda_{\ell}(h_{\ell})=0.
\]
Consequently:
\begin{enumerate}
\item[\rm (a)] if $p\in\Rplus$, then $p=0$;
\item[\rm (b)] the classical Yang-Baxter zero set in every positive marked slice $S_{H}$ is $\{0\}$;
\item[\rm (c)] arbitrary components in $\Lambda^{2}\fb$, $\Lambda^{2}V$, and $\fz\wedge\g$ cannot cancel a non-null rank-one block;
\item[\rm (d)] if $\dim\fb=1$, every triangular tensor has zero $\fb\wedge V$ component, at every frequency pattern.
\end{enumerate}
\end{theorem}

\begin{proof}
Use the exterior decomposition induced by $\g=\fb\oplus V\oplus\fz$ and write uniquely

\[
r=s+p+q+c,\qquad s\in\Lambda^{2}\fb,\ q\in\Lambda^{2}V,\ c\in\fz\wedge\g.
\]
Fix $\ell$. If $p_{\ell}\ne 0$, pick a factorisation $p_{\ell}=h_{\ell}\wedge u_{\ell}$ with $h_{\ell}\ne 0$ and $u_{\ell}\ne 0$. The planewise identity \eqref{eq:planewise1} yields
\begin{equation}\label{eq:block-projection}
[p_{\ell},p_{\ell}]=2\lambda_{\ell}(h_{\ell})\|u_{\ell}\|^{2}\,h_{\ell}\wedge\Omega_{\ell}\in\fb\wedge\Lambda^{2}P_{\ell}.
\end{equation}
We claim that no other term in $[r,r]$ contributes to $\fb\wedge\Lambda^{2}P_{\ell}$. Indeed:
\begin{itemize}[leftmargin=1.6em,itemsep=1pt]
\item brackets between distinct blocks $p_{i}$ and $p_{j}$ lie in $\fb\wedge P_{i}\wedge P_{j}$ (Lemma~\ref{lem:planewise});
\item $[p,q]\in\Lambda^{3}V\oplus(\fb\wedge V\wedge\fz)$ and $[q,q]\in\fz\wedge\Lambda^{2}V$;
\item every bracket involving $c$ retains a factor in $\fz$ by centrality;
\item $[s,s]=0$ because $\fb$ is abelian, and $[s,p]\in\Lambda^{2}\fb\wedge V$;
\item decompose $q=\sum_{a}q_{aa}+\sum_{i<j}q_{ij}$ with $q_{aa}\in\Lambda^{2}P_{a}=\mathbb R\Omega_{a}$ and $q_{ij}\in P_{i}\wedge P_{j}$; since every rotation $\rho(b)|_{P_{a}}$ preserves $\Omega_{a}$, $[s,q_{aa}]=0$, whereas $[s,q_{ij}]\in\fb\wedge P_{i}\wedge P_{j}$.
\end{itemize}
None of these summands lies in $\fb\wedge\Lambda^{2}P_{\ell}$. Therefore

\[
\pr_{\fb\wedge\Lambda^{2}P_{\ell}}[r,r]=2\lambda_{\ell}(h_{\ell})\|u_{\ell}\|^{2}\,h_{\ell}\wedge\Omega_{\ell}.
\]
If $[r,r]=0$, the right-hand side vanishes; since $h_{\ell}\ne 0$ and $u_{\ell}\ne 0$, this forces $\lambda_{\ell}(h_{\ell})=0$. Thus $p\in\Rzero$. Parts (a)--(c) follow at once, and part (d) is immediate because in dimension one every block has tensor rank at most one and $\ker\lambda_{\ell}=0$.
\end{proof}

\begin{remark}[What fails in higher block rank]\label{rem:scope}
In an oriented orthonormal basis $(e_{\ell},f_{\ell})$, a general block of tensor rank two has the form

\[
p_{\ell}=x_{\ell}\wedge e_{\ell}+y_{\ell}\wedge f_{\ell},\qquad x_{\ell},y_{\ell}\in\fb,
\]
and a direct calculation from \eqref{eq:schouten2} and \eqref{eq:beta} gives
\begin{equation}\label{eq:higher-rank-block}
[p_{\ell},p_{\ell}]=2\bigl(\lambda_{\ell}(x_{\ell})x_{\ell}+\lambda_{\ell}(y_{\ell})y_{\ell}\bigr)\wedge\Omega_{\ell}
+2\lambda_{\ell}\wedge x_{\ell}\wedge y_{\ell},
\end{equation}
where $\lambda_{\ell}$ is viewed as an element of $\fz=\fb^{*}$. The first term no longer lies on a fixed ray and may vanish on nontrivial rank-two data, while the second couples the block to the central exterior sector $\fz\wedge\Lambda^{2}\fb\subset\Lambda^{3}\g$. Both positivity and the support-isolation argument break at the same point, and \eqref{eq:higher-rank-block} identifies the precise obstruction to extending the present method.
\end{remark}

\begin{remark}[Modified Yang-Baxter equations]\label{rem:modified}
The proof of Theorem~\ref{thm:obstruction} uses only the planewise projections. The same conclusion therefore holds for $[r,r]=\Theta$ whenever $\pr_{\fb\wedge\Lambda^{2}P_{\ell}}\Theta=0$ for every $\ell$. A modified classical Yang-Baxter equation of this shape inherits the rank-one support obstruction. If instead the prescribed invariant three-tensor has a nonzero radial projection, no support constraint follows from the argument.
\end{remark}

\begin{remark}[Consequence for Yang-Baxter deformations]\label{rem:physics}
A homogeneous Yang-Baxter deformation starts from a skew tensor satisfying \eqref{eq:cybe}. For oscillator and Nappi-Witten algebras, Theorem~\ref{thm:obstruction}(d) removes the entire $h\wedge V$ sector from that input, uniformly in frequencies. The algebraic search for triangular deformations may therefore be restricted from the outset to

\[
r\in\Lambda^{2}V+\fz\wedge\g.
\]
This is compatible with the explicit Nappi-Witten and $H_{4}$ analyses of \cite{KyonoYoshida2016,DemulderEtAl2021} and with the general framework of \cite{BorsatoWulff2019,HoareLacroix2020}; it is a support-level constraint on $r$, not a claim about the conformal invariance, integrability, or physical equivalence of the remaining deformations.
\end{remark}

\section{Two-plane fibres and oscillator comparison}\label{sec:twoplane}

The case of exactly two active planes exhausts the qualitative angular behaviours and admits closed-form calculations, providing a concrete illustration of Theorems~\ref{thm:structure} and~\ref{thm:monomial}.

\subsection{The two-plane trichotomy}

\begin{theorem}[Two-plane trichotomy]\label{thm:trichotomy}
Let $m=2$ and $Y=t_{1}E_{1}(H)+t_{2}E_{2}(H)$ with $t_{1},t_{2}>0$. Exactly one of the following holds.
\begin{enumerate}
\item[\rm (i)] \emph{(No interaction.)} If $\{1,2\}$ is not an edge of $\Gamma_{H}$, then $A_{H,Y}=\varnothing$, $F_{Y}$ is constant, $\rk dF_{Y}=0$, and $K_{A}=\mathbb T^{2}$.
\item[\rm (ii)] \emph{(Cancellation.)} If $\{1,2\}$ is an edge and $\hh_{2}=\tau\hh_{1}$ with $\tau\in\{\pm 1\}$, then $A_{H,Y}=\{\pm(e_{1}-\tau e_{2})\}$, $\rk dF_{Y}=1$, and

\[
K_{A}=\{(\theta_{1},\theta_{2})\in\mathbb T^{2}:\theta_{1}-\tau\theta_{2}=0\}.
\]
In scalar coordinates $y_{+},y_{-}$ on the two surviving coefficient lines, the complexified image is $y_{+}y_{-}=a_{+}a_{-}$.
\item[\rm (iii)] \emph{(No cancellation.)} If $\{1,2\}$ is an edge and no cancellation occurs, then $A_{H,Y}=\{\pm e_{1}\pm e_{2}\}$, $\rk dF_{Y}=2$, and $K_{A}=\{(0,0),(\pi,\pi)\}\cong\mathbb Z/2\mathbb Z$. With coordinates indexed by sign pairs, the complexified image is $y_{++}y_{--}=a_{++}a_{--}$ and $y_{+-}y_{-+}=a_{+-}a_{-+}$.
\end{enumerate}
The joint exponent polytopes in the three cases are respectively $\varnothing$, a line segment, and the square $\conv\{(\pm 1,\pm 1)\}$.
\end{theorem}

\begin{proof}
Case (i) is immediate from \eqref{eq:coeff}. In cases (ii)--(iii), Lemma~\ref{lem:cancellation} determines which coefficients survive; rank and kernel come from Theorem~\ref{thm:monomial} and Corollary~\ref{cor:graph}. The displayed binomials are bases of the respective relation lattices.
\end{proof}

\begin{example}[A resonant oscillator pair]\label{ex:resonant}
Let $\fb=\mathbb Rh$, $h^{*}(h)=1$, and take two equal positive frequencies $\lambda_{1}=\lambda_{2}=\nu h^{*}$ with $\nu>0$ and marking $H=(h,h)$. Then $\hh_{1}=\hh_{2}=h/\nu$, so the unique edge is a cancellation edge of type $\tau=1$. For every full-support radial value,

\[
A_{H,Y}=\{\pm(e_{1}-e_{2})\},\qquad K_{A}=\{(\theta_{1},\theta_{2}):\theta_{1}=\theta_{2}\}.
\]
The angular map depends only on the relative phase $\theta_{1}-\theta_{2}$; its image is a circle, and its complexification is the translated hyperbola $y_{+}y_{-}=a_{+}a_{-}$. This is the simplest resonant configuration excluded by a distinct-frequency hypothesis and displays explicitly the rank drop of Theorem~\ref{thm:monomial}.
\end{example}

\begin{figure}[t]
\centering
\resizebox{\textwidth}{!}{%
\begin{tikzpicture}[>=Latex,line cap=round,line join=round,font=\scriptsize]
  \begin{scope}[xshift=0cm]
    \node[font=\small\bfseries] at (2.15,3.05) {(a) radial cone};
    \fill[MidnightBlue!14] (0.45,0.45)--(3.55,2.35)--(3.75,0.80)--cycle;
    \draw[->,very thick,MidnightBlue] (0.45,0.45)--(3.80,2.55) node[above left] {$E_{1}(H)$};
    \draw[->,very thick,MidnightBlue] (0.45,0.45)--(4.00,0.82) node[right] {$E_{2}(H)$};
    \node[left] at (0.45,0.45) {$0$};
    \node[MidnightBlue,font=\normalsize] at (2.65,1.25) {$C_{H}$};
  \end{scope}
  \draw[gray!55] (4.45,0.15)--(4.45,3.15);
  \begin{scope}[xshift=4.85cm]
    \node[font=\small\bfseries] at (2.15,3.05) {(b) radial fibre};
    \draw[very thick] (2.15,1.55) ellipse (1.75 and 0.88);
    \draw[very thick] (2.15,1.55) ellipse (0.68 and 0.32);
    \draw[dashed,MidnightBlue,thick,->] (0.48,1.55) arc[start angle=180,end angle=350,x radius=1.67,y radius=0.72];
    \draw[MidnightBlue,thick,->] (2.15,1.23) arc[start angle=-90,end angle=245,x radius=0.68,y radius=0.32];
    \node[align=center] at (2.15,0.22) {$T_{Y}=S^{1}_{\sqrt{t_{1}}}\times S^{1}_{\sqrt{t_{2}}}$};
  \end{scope}
  \draw[gray!55] (9.35,0.15)--(9.35,3.15);
  \begin{scope}[xshift=9.70cm]
    \node[font=\small\bfseries] at (2.55,3.05) {(c) two-plane exponent polytopes};
    \draw[dashed,gray] (0.15,0.72) rectangle (1.25,1.82);
    \node[MidnightBlue,font=\large] at (0.70,1.27) {$\varnothing$};
    \node[align=center] at (0.70,0.35) {empty};
    \draw[very thick,MidnightBlue] (1.72,1.27)--(2.88,1.27);
    \fill[MidnightBlue] (1.72,1.27) circle (1.5pt) (2.88,1.27) circle (1.5pt);
    \node[align=center] at (2.30,0.35) {segment};
    \draw[very thick,MidnightBlue] (3.40,0.78) rectangle (4.65,1.76);
    \fill[MidnightBlue] (3.40,0.78) circle (1.5pt) (3.40,1.76) circle (1.5pt)
      (4.65,0.78) circle (1.5pt) (4.65,1.76) circle (1.5pt);
    \node[align=center] at (4.03,0.35) {square};
  \end{scope}
\end{tikzpicture}%
}
\caption{The geometry of the restricted Schouten square: the simplicial radial cone $C_{H}$, a product-of-circles radial fibre $T_{Y}$, and the empty, segment, and square joint exponent polytopes of the two-plane trichotomy.}
\label{fig:geometry}
\end{figure}
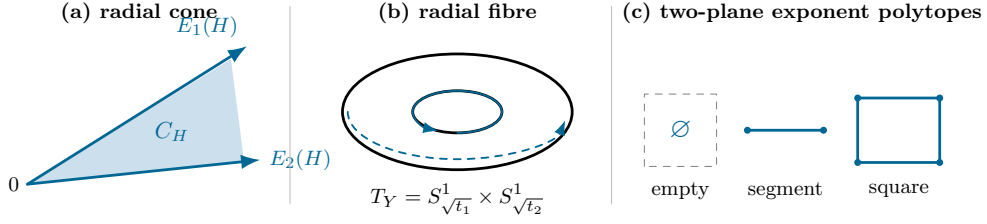

\subsection{A closed-form real fibre}

Take $\fb=\mathbb Rb_{1}\oplus\mathbb Rb_{2}$, $\lambda_{i}=b_{i}^{*}$, and choose the marking $H'=(b_{1}+b_{2},b_{2})$. Then
\begin{equation}\label{eq:muHprime}
\mu_{H'}(r)=2\,b_{2}\wedge u_{1}\wedge J_{2}u_{2},
\end{equation}
so the angular type is case (iii) of Theorem~\ref{thm:trichotomy}.

\begin{proposition}[Closed-form two-plane fibre]\label{prop:closedfibre}
Fix $t_{1},t_{2}>0$ and put $Y=t_{1}E_{1}(H')+t_{2}E_{2}(H')$. In oriented orthonormal bases $(e_{\ell},f_{\ell})$ write $\Theta_{\mathrm{mix}}=b_{2}\wedge\sum_{p,q}m_{pq}\,v_{1,p}\wedge v_{2,q}$, where $(v_{\ell,1},v_{\ell,2})=(e_{\ell},f_{\ell})$, and set $M=(m_{pq})$. The equation $[r,r]=Y+\Theta_{\mathrm{mix}}$ has a solution $r\in T_{Y}$ iff
\begin{equation}\label{eq:detM}
\det M=0,\qquad \|M\|_{F}=2\sqrt{t_{1}t_{2}}.
\end{equation}
The solution set then consists of exactly two points, exchanged by the simultaneous half-turn $(\theta_{1},\theta_{2})\mapsto(\theta_{1}+\pi,\theta_{2}+\pi)$.
\end{proposition}

\begin{proof}
Writing $u_{\ell}=\sqrt{t_{\ell}}(\cos\theta_{\ell}e_{\ell}+\sin\theta_{\ell}f_{\ell})$ and expanding \eqref{eq:muHprime} gives

\[
M=2\sqrt{t_{1}t_{2}}\begin{pmatrix}\cos\theta_{1}\\ \sin\theta_{1}\end{pmatrix}
\begin{pmatrix}-\sin\theta_{2}&\cos\theta_{2}\end{pmatrix}.
\]
Thus $M$ has rank one and Frobenius norm $2\sqrt{t_{1}t_{2}}$. Conversely, every real rank-one $2\times 2$ matrix of that norm factors as $2\sqrt{t_{1}t_{2}}\,pq^{T}$ for unit vectors $p,q$, uniquely up to $(p,q)\mapsto(-p,-q)$; the two factorisations give the two phase points exchanged by the simultaneous half-turn.
\end{proof}

\subsection{Oscillator algebras: a genericity-free support theorem}

Specialise to $\fb=\mathbb Rh$, write $\nu_{\ell}=\lambda_{\ell}(h)>0$, and let $z\in\fb^{*}$ be dual to $h$. The brackets are
\begin{equation}\label{eq:oscillator-bracket}
[h,e_{\ell}]=\nu_{\ell}f_{\ell},\qquad [h,f_{\ell}]=-\nu_{\ell}e_{\ell},\qquad [e_{\ell},f_{\ell}]=\nu_{\ell}z.
\end{equation}
This is the oscillator algebra with frequencies $\nu_{1},\dots,\nu_{m}$; the case $m=1$, $\nu_{1}=1$ is the Nappi-Witten algebra \cite{NappiWitten1993}.

\begin{proposition}[Oscillator support rigidity]\label{prop:oscillator}
For arbitrary positive frequencies $\nu_{1},\dots,\nu_{m}$, every solution $r\in\Lambda^{2}\g$ of the classical Yang-Baxter equation on the oscillator algebra \eqref{eq:oscillator-bracket} has zero component in $h\wedge V$.
\end{proposition}

\begin{proof}
Since $\dim\fb=1$, every block of the $\fb\wedge V$ component has tensor rank at most one, and every $\lambda_{\ell}$ has trivial kernel. Theorem~\ref{thm:obstruction}(d) applies.
\end{proof}

Under the distinct-frequency hypothesis of \cite{BoucettaMedina2011} Proposition~\ref{prop:oscillator} is the support projection of the complete classification. Its content is the removal of every genericity assumption: the argument requires neither distinct frequencies nor the exclusion of additive resonances. The radial and angular results of Sections~\ref{sec:structure}--\ref{sec:monomial} then give further information about nonzero level sets of the Schouten square, including the resonant cancellation stratum of Example~\ref{ex:resonant}. Related Poisson and bialgebra structures on oscillator algebras appear in \cite{AlbuquerqueEtAl2021,BallesterosHerranz1996,BallesterosHerranz1997}.

\section{Concluding remarks}\label{sec:conclusion}

The restricted Schouten square on a positive marked slice carries a rigid two-level toric geometry. Its radial component is an injective linear image of a toric moment map with simplicial image $C_{H}$ and sharp quadratic coercive estimate $\kappa_{0}(H)$; its angular component is a Laurent monomial map whose exponent lattice $L_{A}$ makes every fibre invariant algorithmic via Smith normal form. This separation delivers more than a non-vanishing statement: it describes the image and all affine fibres of $\Sigma_{H}$, proves properness and compactness, produces a fixed-support constant-rank stratification, and yields defining binomial equations for the Zariski closure of the complexified angular image. Cancellation is the sole mechanism of angular rank drop, and its strata are explicit semialgebraic submanifolds of the marking space.

The Yang-Baxter consequence is best read as support rigidity. Once the $\fb\wedge P_{\ell}$ blocks have tensor rank at most one, triangularity of a completely arbitrary bivector forces every marker detected by the corresponding weight to disappear, regardless of $\Lambda^{2}\fb$, $\Lambda^{2}V$, and central-wedge components. For $\dim\fb=1$ the rank-one condition is automatic; the $h\wedge V$ sector is therefore excluded for every frequency pattern, providing a genericity-free algebraic reduction for the classical $r$-matrices entering Yang-Baxter deformations of oscillator and Nappi-Witten targets.

Two structural extensions are natural, and \eqref{eq:higher-rank-block} identifies their obstacles. First, higher tensor rank in a two-plane block replaces the positive ray by the vector-valued expression

\[
2\bigl(\lambda_{\ell}(x_{\ell})x_{\ell}+\lambda_{\ell}(y_{\ell})y_{\ell}\bigr)\wedge\Omega_{\ell}
+2\lambda_{\ell}\wedge x_{\ell}\wedge y_{\ell},
\]
which is neither confined to a fixed ray nor supported inside $\fb\wedge\Lambda^{2}V$. Any extension must simultaneously control the loss of radial positivity and the appearance of central three-form components. Second, if a rotation block is a higher-dimensional Hermitian space, $u\wedge Ju$ is no longer confined to a one-dimensional area line; the radial image becomes matrix-valued rather than simplicial, and the angular phase space is no longer described by a product of one-dimensional characters. Both are genuine structural transitions rather than routine variations of the present analysis.

\section*{Data availability}
No datasets were generated or analysed in this study.

\end{document}